\documentclass[10pt,reqno]{amsart}

\usepackage{amsmath,amsfonts,amsthm,amssymb,amsxtra}
\usepackage{amsxtra, amssymb, mathrsfs}
\usepackage[usenames,dvipsnames]{color}

\newtheorem{theorem}{Theorem}[section]
\newtheorem{proposition}[theorem]{Proposition}

\newtheorem{lemma}[theorem]{Lemma}
\newtheorem{corollary}[theorem]{Corollary}

\theoremstyle{definition}

\newtheorem{assumption}[theorem]{Assumption}

\theoremstyle{remark}
\newtheorem{remark}[theorem]{\bf Remark}

\numberwithin{equation}{section}

\newcommand{\B}{\mathscr{B}}
\newcommand{\BB}{\mathcal{B}}

\newcommand{\C}{\mathbb{C}}

\newcommand{\G}{\mathcal{G}}
\newcommand{\h}{\mathcal{H}}

\newcommand{\U}{\mathcal{U}}

\newcommand{\pd}{\partial}
\newcommand{\eps}{\varepsilon}

\newcommand{\N}{\mathbb{N}}

\newcommand{\R}{\mathbb{R}}

\newcommand{\w}{\mathcal{W}}
\newcommand{\Z}{\mathbb{Z}}
\newcommand{\rig}{\big\rangle}
\newcommand{\lef}{\big\langle}
\newcommand{\sgn}{\mathop{\mathrm{sign}}}

\begin{document}

\title[Dispersive estimates for Aharonov-Bohm operators]
{Weighted dispersive estimates for two-dimensional Schr\"odinger operators with Aharonov-Bohm magnetic field}

\thanks{G.G. has been partially supported by the MIUR-PRIN 2009 grant ``Metodi di viscosit\`a, geometrici e di controllo per modelli diffusivi nonlineari''. H.K. has been partially supported by the MIUR-PRIN 2010-11 grant for the project  ``Calcolo delle Variazioni''. Both authors acknowledge the support of Gruppo Nazionale per l'Analisi Matematica, la Probabilit\`a   e le loro Applicazioni (GNAMPA) of the Istituto Nazionale di Alta Matematica (INdAM).}

\author {Gabriele Grillo}

\address {Gabriele Grillo, Dipartimento di Matematica, Politecnico di Milano\\
              Piazza Leonardo da Vinci 32, 20133 Milano, Italy\\
              }

\email {gabriele.grillo@polimi.it}

\author {Hynek Kova\v{r}\'{\i}k}

\address {Hynek Kova\v{r}\'{\i}k
              DICATAM, Sezione di Matematica, Universit\`a degli studi di Brescia\\
              Via Branze, 38, 25123 Brescia, Italy}

\email {hynek.kovarik@ing.unibs.it}


\begin {abstract}
We consider two-dimensional Schr\"odinger operators $H$ with an Aharonov-Bohm magnetic field and an additional electric potential.
We obtain an explicit leading term of the asymptotic expansion of the unitary group $e^{-i t H}$ for $t\to\infty$ in weighted $L^2-$spaces. In particular, we show that the magnetic field improves the decay of $e^{-i t H}$ with respect to the unitary group of non-magnetic Schr\"odinger operators, and that the decay rate in time is determined by the magnetic flux.

\end{abstract}

\maketitle


\section{\bf Introduction}

Long time behavior of propagators of  Schr\"odinger operators is known to be closely related to the spectral properties of their generator near the threshold of the continuous spectrum. For example, if zero is a regular point of a Schr\"odinger operator $-\Delta +V$ in $L^2(\R^n)$ in the sense of \cite{JK}, then in a suitable operator topology
\begin{equation} \label{rn}
e^{-i t (-\Delta+V)} P_{ac} = \mathcal{O}(t^{-n/2}) \qquad t\to \infty,
\end{equation}
provided the electric potential $V:\R^n \to \R$ decays fast enough. Here $P_{ac}$ denotes the projection onto the absolutely continuous spectral subspace of $-\Delta+V$. For $n=3$ such dispersive estimates were established in \cite{ra, je, JK} in weighted $L^2-$spaces. For corresponding $L^1 \to L^\infty$ bounds we refer to \cite{jss, gsch,wed} in the case $n=1,3$, and to \cite{sch} in the case $n=2$.  Situations in which zero is not a regular point of $-\Delta +V$   are studied in great generality in \cite{JK}, see also  \cite{es1, es2, rs}.  For a recent survey on the existing dispersive estimates for Schr\"odinger operators and a thorough bibliography, see \cite{sch2}.


On the other hand, very little is known about estimates of type \eqref{rn} for magnetic Schr\"odinger operators, where $-\Delta+V$ is replaced by $(i\nabla +A)^2+V$ with a vector potential $A$.  Recently it was shown in \cite{kk} that for $n=3$ equation \eqref{rn} can be extended to magnetic Schr\"odinger operators, in suitable weighted $L^2-$spaces, under certain decay and regularity conditions on $A$ and $V$. As for the case $n=2$,  it was proved, see \cite{fffp},  that \eqref{rn} holds true when $V=0$ and $A=A_{ab}$ is the vector potential generating the so-called Aharonov-Bohm magnetic field, see equation \eqref{ab-potential} below.

\smallskip

The aim of this paper is to point out  the {\it diamagnetic effect} on the dispersive estimates \eqref{rn} in the case $n=2$. More precisely, we want to show that the magnetic field improves the decay rate of the propagator in dimension two. This is partially motivated by the following fact \cite{mu, sch2, go, eg, wed2}: if a Schr\"odinger operator $-\Delta+V$ in $L^2(\R^n)$ with $ n=1,2$ does not have a resonance at zero energy, then the time decay of $e^{-i t (-\Delta+V)}$, considered in suitable topology, is faster than the one predicted by \eqref{rn}. Since a magnetic field in $\R^2$ generically removes the resonance at zero energy, see \cite{lw, timo}, it is natural to expect a {\it faster} time decay for propagators generated by magnetic Schr\"odinger operators with respect to the non-magnetic ones.

We will prove this conjecture in the case of the Aharonov-Bohm magnetic field, that is for operators of the type $(i\nabla +A_{ab})^2+V$. In particular, we will show, under suitable assumptions on $V$, that in certain weighted $L^2-$spaces the associated propagator  decays as
\begin{equation} \label{imp-decay}
t^{-1-\min_{k\in\Z}|k-\alpha|}\, , \qquad t\to \infty,
\end{equation}
where $\alpha$ is the total flux of the Aharonov-Bohm field, see Theorem \ref{thm-2} for details. This is to be compared with the $t^{-1}$ decay rate of non-magnetic Schr\"odinger propagators, see \eqref{rn}.  In the special case $V=0$ we obtain a stronger result which provides a point wise upper bound on the integral kernel of $e^{-it (i\nabla +A_{ab})^2}$, see Theorem \ref{thm-1} and Corollary \ref{cor}.

Note that the decay rate of  the unitary group generated by $(i\nabla +A_{ab})^2+V$ is proportional to the distance between the magnetic flux and the set of integers, see \eqref{imp-decay}. This is expected since an Aharonov-Bohm field with an integer flux can be gauged away and therefore does not affect the spectral properties of the corresponding generator. Our main results, Theorem \ref{thm-1} and Theorem \ref{thm-2} are presented in section \ref{sec-results}. The proofs of the main results are given in section \ref{sec-proofs}.

Although our paper deals only with the case of the Aharonov-Bohm field, we believe that a similar improved time decay should occur for a much wider class of magnetic fields, see Remark \ref{rem-general} for further discussion.


\section{\bf Preliminaries}
\label{sec-prelim}
\noindent The vector potential
\begin{equation} \label{ab-potential}
A_{ab}(x) = (A_1(x), A_2(x))= \frac{\alpha}{|x|^2}\,  (-x_2\, ,\, x_1) \quad \text{on } \quad \R^2\setminus \{0\},
\end{equation}
generates the Aharonov-Bohm magnetic field which is fully characterized by its constant flux $\alpha$.
It is well-known, see e.g. \cite{at}, that the operator
\begin{equation} \label{c0}
(i\nabla +A_{ab})^2 \qquad \text{on} \quad C_0^\infty(\R^2\setminus\{0\})
\end{equation}
is not essentially self-adjoint and has deficiency indices $(2,2)$. Consequently, it admits infinitely many self-adjoint extensions. In this paper we will work with the Friedrichs extension of \eqref{c0} which we denote by $H_\alpha$. In order to define the latter we introduce a function space $W^{1,2}_\alpha(\R^2)$ given by the closure of $C_0^\infty(\R^2\setminus\{0\})$ with respect to the norm
$$
\| u\|_{L^2(\R^2)} + \| (i\nabla+A_{ab})\, u\|_{L^2(\R^2)}.
$$
The quadratic form
\begin{equation}
Q_\alpha [u]= \| (i \nabla+A_{ab})\, u\|^2_{L^2(\R^2)}
\end{equation}
with the form domain $W^{1,2}_\alpha(\R^2)$ is then closed and generates a unique non-negative self-adjoint operator $H_\alpha$ in $L^2(\R^2)$. Next we introduce an additional electric potential $V:\R^2\to \R$ and consider the Schr\"odinger operator
\begin{equation} \label{H}
H = H_\alpha+V   \qquad \text{in} \ \  L^2(\R^2),
\end{equation}
see section \ref{aaa} for a precise definition of $H$.

\begin{remark} \label{rem-gauge}
The operators $H_\alpha+V$ and $H_{m+\alpha}+V$ are unitarily equivalent for any $m\in\Z$. Indeed, in view of  equation \eqref{q-form} we have
$$
H_{m+\alpha}+V = \, \U_m\, (H_\alpha+V)\ \U_m^{-1}, \qquad m\in\Z,
$$
where $\U_m: L^2(\R^2) \to L^2(\R^2) $ is the unitary operator acting as $\U_m\, u = e^{i m\theta}\, u$. We may therefore assume without loss of generality that
$$
0 < | \alpha| \leq \frac 12, \qquad \text{which \ implies} \qquad  \min_{k\in\Z}|k-\alpha| = |\alpha|.
$$
\end{remark}

\subsection{Notation}
Given $s\in\R$ we denote
$$
L^2(\R^2, w^s) = \{ u : \|  w^{s}  u\|_{L^2(\R^2)} < \infty \}, \qquad \|u\|_{0,s} :=  \|  w^{s}  u\|_{L^2(\R^2)},
$$
where $w(x) =(1+|x|^2)^{1/2}$. For $x=(x_1, x_2)\in\R^2$ and  $y=(y_1, y_2) \in \R^2$  we will often use the polar coordinates representation
\begin{equation} \label{polar-repr}
x_1 + i x_2  = r  e^{i \theta}, \quad y_1 + i y_2 =r'  e^{i \theta'},  \qquad r,r' \geq 0, \quad \  \theta, \theta' \in [0, 2\pi).
\end{equation}
Given $R>0$ and a point $x\in\R^2$ we denote by $B(x,R)\subset\R^2$ the open ball with radius $R$ centered in $x$.
Let $\B(s,s')$ be the space of bounded linear operators from $L^2(\R^2, w^s)$ to $L^2(\R^2, w^{s'})$.
By $\|K\|_{\B(s,s')}$ we denote the norm of a bounded linear operator $K$ in $\B(s,s')$. 
The scalar product in a Hilbert space $\mathscr{H}$ will be denoted by $\langle \cdot \, , \cdot \rangle_{\mathscr{H}}$. Finally, we denote $\R_+ = (0,\infty)$ and $\C_+ = \{z\in\C\, :\, {\rm Im}\, z >0\}$.

\medskip


\section{\bf Main results}
\label{sec-results}

\subsection{ Time decay for $e^{-it H_\alpha}$} We say that $e^{-i t H_\alpha}(x,y), \ x,y \in \R^2$ is an integral
kernel of the operator $e^{- it H_\alpha}$ if
\begin{equation} \label{L2-kernel}
\big(e^{- it H_\alpha}\, u\big) (x) = \int_{\R^2} e^{-i t H_\alpha}(x,y) \, u(y)\, dy
\end{equation}
holds for any $u\in L^2(\R^2)$ with compact support.

\begin{theorem} \label{thm-1}
Let $| \alpha| \leq 1/2$. There exists a kernel $e^{-i t H_\alpha}(x,y)$ of $e^{- it H_\alpha}$ such that for all $x,y\in\R^2$ it holds
\begin{align} \label{limit-1}
\lim_{t\to \infty}\, (it)^{1+ |\alpha|}\, e^{-i t H_\alpha}(x,y) & =
\frac{1}{4\pi \Gamma(1+| \alpha|)}\, \Big(\frac{ r
r'}{4}\Big)^{|\alpha|} & \quad \text{if\, \, } |\alpha| < 1/2, \\  \label{limit-2}
\lim_{t\to \infty}\, (it)^{\frac 32}\, e^{-i t H_\alpha}(x,y) & =
\frac{1}{4\pi \Gamma(3/2)}\, \Big(\frac{r r'}{4}\Big)^{\frac
12}\, (1+e^{\mp i(\theta-\theta')}) & \quad \text{if\, \, } \alpha = \pm 1/2.
\end{align}

\smallskip

\noindent Moreover, there is a constant $C$ such that for all $t>0$ and all $x,y \in\R^2$
\begin{equation} \label{point-upperb}
|\, e^{-i\, t H_\alpha}(x,y)\, | \, \leq\, C\, \min\big\{ t^{-1},\  (r r') ^{|\alpha|}\, \ t^{-1-|\alpha|}\, \big\} .
\end{equation}
\end{theorem}

\medskip

\begin{remark}
In \cite[Thm.1.3 \& Cor.1.7]{fffp} it was shown that
\begin{equation*}
\sup_{x,y \in \R^2}\, |\, e^{-i\, t H_\alpha}(x,y)\, | \, \leq\, {\rm const} \ t^{-1}.
\end{equation*}
This obviously yields the first part of \eqref{point-upperb}. Moreover, the decay rate $t^{-1}$ in the above estimate is sharp. However, the second term on the right hand side of  \eqref{point-upperb} shows that for fixed $x$ and $y$ the kernel $ e^{-i\, t H_\alpha}(x,y)$ decays faster than $t^{-1}$.
\end{remark}

\noindent Inequality \eqref{point-upperb} implies

\begin{corollary} \label{cor}
There exists $C>0$ such that for all $t>0$ we have
\begin{align*}
\| e^{-i\, t H_\alpha}\, u \|_{0,-s} \ & \leq\ C\ t^{-1-|\alpha|}\, \|u\|_{0,s} \qquad \qquad \ \ \forall\ u \in L^2(\R^2, w^s), \quad s > 1+|\alpha|,\\
\| w^{-s} e^{-i\, t H_\alpha}\, u \|_\infty \ & \leq\ C\ t^{-1-|\alpha|}\, \|w^s\, u\|_{L^1(\R^2)} \qquad \forall\ u \in L^1(\R^2, w^s), \quad  s \geq |\alpha|,
\end{align*}
where $L^1(\R^2, w^s) = \{ u : \|  w^{s}\,  u\|_{L^1(\R^2)} < \infty \}$.
\end{corollary}

\smallskip

\begin{remark}
Improved time decay of semi-groups generated by two-dimensional magnetic Schr\"odinger operators was recently studied in \cite{k11, kr}.
\end{remark}

\medskip

\subsection{ Time decay for $e^{-it (H_\alpha+V)}$}\label{aaa} In this case we will introduce another operator norm associated to the weight function $\rho:\R^2\to \R$ given by $\rho(x) = \rho(|x|) =  e^{|x|^4}$. Let
$$
L^2(\R^2, \rho) = \{ u : \|  \rho^{1/2} \, u\|_{L^2(\R^2)} < \infty \}, \qquad \|u\|_{\rho} :=  \|  \rho^{1/2}\,  u\|_{L^2(\R^2)}.
$$
We denote by $\B(\rho, \rho^{-1})$ the space of bounded linear operators from $L^2(\R^2,\rho)$ to $L^2(\R^2,\rho^{-1})$.
Consequently,  we will denote by $\|K\|_{\B(\rho, \rho^{-1})}$ the norm of $K\in\B(\rho, \rho^{-1})$.
Recall that we are assuming $ 0 < |\alpha| \leq 1/2$ and let $G_0$ be the integral operator with the kernel
\begin{equation} \label{g0}
G_0(x,y)  = \frac {1}{4\pi}\, \sum_{m\in \Z} \, \frac{\, e^{im(\theta-\theta')}}{|\alpha+m|}\  \Big(\min \Big\{ \frac{r}{r'}\, ,\, \frac{r'}{r}\, \Big\} \Big)^{|\alpha+m|}.
\end{equation}
Note that the series on the right hand side of \eqref{g0} converges for all $x\neq y$. Moreover, from equation \eqref{hs-norm} below it easily follows that $G_0 \in \B(\rho, \rho^{-1})$. As for the potential $V$, we suppose that it satisfies the following

\begin{assumption} \label{ass-V}
The function $V: \R^2\to \R$ is bounded and compactly supported.  Moreover, for any $ u\in L^2(\R^2, \rho^{-1})$ it holds
\begin{equation} \label{V-as2}
 u+ G_0 V \, u = 0 \quad \Rightarrow \quad u=0.
\end{equation}
\end{assumption}

\noindent Condition \eqref{V-as2} is crucial as it excludes the possibility that $0$ is a resonance of $H$.

\smallskip

\noindent The motivation for the assumption that $V$ has compact support is twofold. On one hand, it guarantees that $V\in \B(\rho^{-1}, \rho)$ which is needed in \eqref{V-as2}. On the other hand, the compactness of the support of $V$ will play an important role in the proof of absence of positive eigenvalues of the operator $H_\alpha+V$, see Proposition \ref{prop-high}.

Note that since we don't have at our disposal a class of slowly decaying potentials for which $H_\alpha +V$ has no eigenvalues, contrary to the non-magnetic situation, we cannot a priori follow the general procedure invented
in \cite{rs} in order to include potentials with power-like decay at infinity.

\medskip

\noindent To proceed we note that since $V$ is bounded, the operator $H= H_\alpha +V$ is self-adjoint on the domain of $H_\alpha$, by Rellich's theorem. Let $P_c= P_{c}(H)$ denote the projection onto the continuous spectral subspace of $H$. We have

\begin{theorem} \label{thm-2}
Let $ 0 < |\alpha| \leq 1/2$ and let $V$ satisfy assumption \ref{ass-V}. Denote by $\G_j\in \B(\rho, \rho^{-1}),\, j=1,2,$ the operators with  integral kernels
\begin{align*}
\mathcal{G}_1(x,y) & = \frac{1}{4 \pi \Gamma(1+|\alpha|)}\,  \Big(\frac {r r'}{4}\Big) ^{|\alpha|} 
\\
\G_2(x,y) & = e^{i (\theta'-\theta) \sgn\alpha}\, \, \frac{\Gamma(1+|\alpha|)}{4\pi \Gamma^2(1+|\alpha-\sgn\alpha|)}\,  \Big(\frac {r r'}{4}\Big) ^{|\alpha-\sgn\alpha|}   .
\end{align*}
Then, as $t\to\infty$
\begin{align}
e^{-i t H}\, P_c & =   (it)^{-1-|\alpha|}\, (1+G_0 V)^{-1}\, \G_1\, (1+V G_0 )^{-1} + o (\, t^{-1-|\alpha|})
  \quad  & \text{if} \quad |\alpha| < 1/2, \label{limit-3}\\
 e^{-i t H}\, P_c & = (it)^{-\frac 32}\, (1+G_0 V)^{-1}\, (\G_1+\G_2)\, (1+V G_0)^{-1} + o (\, t^{-\frac 32})     \quad & \text{if} \quad |\alpha| = 1/2  \label{limit-4}
\end{align}
in $\B(\rho, \rho^{-1})$.
In particular, there exists a constant $C$ such that for all $t>0$ and all $u \in L^2(\R^2, \rho)$ we have
\begin{equation} \label{weight-estim}
\| \, e^{-i\, t H}\, P_c\, u \|_{\rho^{-1} }\, \leq\, C\ t^{-1-|\alpha|}\, \|u\|_{\rho}.
\end{equation}
\end{theorem}

\smallskip

\begin{remark}
The reason why we work with the weighted spaces $L^2(\R^2, \rho)$ and not with $L^2_s(\R^2)$ is technical. Our goal is to keep track of the diamagnetic effect on $e^{-i\, t H}$, i.e. of the improved time decay caused by the magnetic field, observed in the case $V=0$. Therefore we employ the perturbation approach with $H_\alpha$ as the free operator. This requires a precise knowledge of the behavior of the resolvent of $H_\alpha$ near the threshold of the spectrum. Since we don't have a simple formula for its integral kernel, contrary to the situation without a magnetic field, we work with the power-series \eqref{serie} below. To make sure that this series converges absolutely in $\B(\rho, \rho^{-1})$ we have to require a very fast growth of the weight function $\rho$. Note that super-polynomially growing weight functions were used to study decay estimates of non-magnetic Schr\"odinger operators already in \cite{ra}.

Nonetheless, we expect that the claim of Theorem \ref{thm-2} remains true also if $L^2(\R^2, \rho)$ and $L^2(\R^2, \rho^{-1})$ are replaced throughout by $L^2(\R^2, w^s)$ and $L^2(\R^2, w^{-s})$ with $s$ large enough.
\end{remark}

\begin{remark}
Assumption \ref{ass-V} ensures that the operators $1+G_0 V$ and $1+V G_0$ are invertible in $\B(\rho^{-1}, \rho^{-1})$ and $\B(\rho, \rho)$ respectively, see the proof of Lemma \ref{lem-exp1} and Remark \ref{dual} for details. Hence $(1+G_0 V)^{-1}\, \G_j\, (1+V G_0 )^{-1}$ are well defined and belong to $\B(\rho, \rho^{-1})$. Note also that in the case $V=0$ equations \eqref{limit-3} and \eqref{limit-4} agree with the asymptotics \eqref{limit-1} and \eqref{limit-2}.
\end{remark}

\begin{remark} \label{rem-general}
Our method does not enable us to extend the above results to a more general class of magnetic fields. On the other hand, decay estimates on the magnetic heat semi-group obtained recently in \cite{k11, kr} suggest that for a sufficiently smooth magnetic field $B=$ rot$\, A$ with a finite total flux $\alpha$ one should be able to observe, in a suitable operator topology, that as $t\to\infty$
\begin{align*}
e^{-it (i\nabla +A)^2} & = \mathcal{O}(t^{-1-\min_{k\in\Z}|k-\alpha|)})    \qquad  \ \ \text{if } \ \alpha \notin\Z, \\
e^{-it (i\nabla +A)^2} & = \mathcal{O}\big(t^{-1}\, \log^{-2} t\big)  \qquad \qquad\quad \text{if } \ \alpha \in \Z.
\end{align*}
This question remains open.
\end{remark}

\section{\bf Partial wave decomposition} \label{ss:pwave}

\noindent The quadratic form associated to $H_\alpha$ in polar coordinates $(r,\theta)$ reads as follows:
\begin{equation} \label{q-form}
Q_\alpha[u] =  \int_0^\infty\! \int_0^{2\pi} \left(|\pd_r u|^2+ r^{-2}
|i\, \pd_\theta u+\alpha u|^2\right) r\, dr d\theta, \quad u\in W^{1,2}_\alpha(\R^2).
\end{equation}
By expanding a given function $u\in L^2(\R_+\times (0,2\pi))$ into a
Fourier series with respect to the basis $\{e^{i
 m\theta}\}_{m\in\Z}$ of $L^2((0,2\pi))$, we obtain
a direct sum decomposition
\begin{equation} \label{sum-gen}
H_\alpha =   \bigoplus_{m\in\Z}  \big( h_m \otimes\mbox{id}\big)
\Pi_m,
\end{equation}
where $h_m$ are operators generated by the closures, in $L^2(\R_+, r
dr)$, of the quadratic forms
\begin{equation} \label{qm}
Q_m [f] = \int_0^\infty\, \Big(|f'|^2+\frac{(\alpha+m)^2}{r^2}\,
|f|^2\Big)\, r\, dr
\end{equation}
defined initially on $C_0^\infty(\R_+)$, and $\Pi_m$ is the projector acting as
$$
(\Pi_m\, u)(r,\theta) = \frac{1}{2\pi}\, \int_0^{2\pi}\,
e^{im(\theta-\theta')}\, u(r,\theta')\, d\theta'.
$$


\section{\bf Resolvent estimates}
\label{sec-aux}

\noindent In order to prove Theorem \ref{thm-2} we will follow the strategy developed in \cite{JK}. This requires precise estimates on the resolvent of $(H-z)^{-1}$ and its derivatives with respect to $z$. Such estimates are the main objects of our interest in this section.Let us denote by $R_0(\alpha, z) = (H_\alpha-z)^{-1}, \, z\in\C$ the resolvent of the free operator $H_\alpha$. From now on, in order to keep in mind that we work with the continuation of the resolvent form the upper-half plane, for any $\lambda\in(0, \infty)$ we will adopt the following notation:
\begin{equation} \label{it-limit}
R_0(\alpha, \lambda) : = \lim_{\eps\to 0+}\,  R_0(\alpha, \lambda+i\eps) ,
\end{equation}
where according to \cite[Prop. 7.3]{IT} the limit is attained locally uniformly in $\lambda$ on $(0,\infty)$ with respect to the norm of $\B(s,-s)$ for any $s>1/2$.


\subsection{Low energy behavior}

Recall that we keep on assuming $ 0 < |\alpha| \leq 1/2$. In order to study the behavior of $R_0(\alpha, \lambda)$ for  $\lambda\to 0$,  $\lambda \in (0,\infty)$, we first consider
the resolvent kernels of the one-dimensional operators $h_m$ associated with quadratic form \eqref{qm}.
We follow \cite[Sect.5]{k11} and consider the operators
\begin{equation} \label{lbeta2}
\h_m = \w\, h_m\, \w^{-1} \qquad \text{in \, \, \, } L^2(\R_+, dr),
\end{equation}
where $\w: L^2(\R_+, r\, dr)\to L^2(\R_+, dr)$ is a unitary mapping
acting as $(\w f)(r) = r^{1/2} f(r)$. Note that $\h_m$ is subject to Dirichlet boundary condition at zero. To find an expression for the resolvent of $\h_m$ we have to solve the generalized eigenvalue equation
$$
\h_m \, u = - u''(r)\, +\,  \frac{(m+\alpha)^2-\frac 14}{r^2}\ u = \lambda\, u,
$$
and in particular we have to find two solutions $\psi(\lambda, r)$ and $\phi(\lambda,r)$ which satisfy the conditions
$\psi(\lambda,0)=0$ and $\phi(\lambda+i\eps, \cdot\, ) \in L^2(1,\infty)$ for $\eps>0$.
By a straightforward calculation, taking into account the asymptotic properties of Bessel functions, namely \cite[Eqs.9.1.3, 9.2.3]{as}, we find that
\begin{align*}
\psi(\lambda, r) & = \sqrt{r}\ J_{|m+\alpha|}(r\sqrt{\lambda}\, ) \\
\phi(\lambda, r) & = \sqrt{r}\  \left( J_{|m+\alpha|}(r\sqrt{\lambda}\, ) + i\, Y_{|m+\alpha|}(r \sqrt{\lambda}\, ) \right),
\end{align*}
where $J_{\nu}$ and $Y_{\nu}$ are the Bessel functions of the first and second kind respectively.
Since the Wronskian of $\psi$ and $\phi$ is equal to $-2 i/\pi$, see \cite[Eq.9.1.16]{as},
the Sturm-Liouville theory shows that for $r<r'$ we have
\begin{align*}
\lim_{\eps\to 0+}(\h_m-\lambda-i\eps)^{-1}(r,r') & = \frac{\pi i}{2}\,  \sqrt{r r' }\,
\,J_{|m+\alpha|}(r\sqrt{\lambda}\, ) \left(
J_{|m+\alpha|}(r'\sqrt{\lambda}\, ) + i\, Y_{|m+\alpha|}(r
'\sqrt{\lambda}\, ) \right) .
\end{align*}
Hence in view of \eqref{lbeta2} it follows that for $\lambda \in (0,\infty)$ and  $r<r'$
\begin{align*}
(h_m-\lambda)^{-1}(r,r') & = \lim_{\eps\to 0+}(h_m-\lambda-i\eps)^{-1}(r,r') = \frac{\pi i}{2}\,
\,J_{|m+\alpha|}(r\sqrt{\lambda}\, ) \big(
J_{|m+\alpha|}(r'\sqrt{\lambda}\, ) + i\, Y_{|m+\alpha|}(r
'\sqrt{\lambda}\, ) \big).
\end{align*}
When $r'\leq r$, then we switch $r'$ and $r$
in the above formula. In the sequel we will assume for definiteness that $r< r'$.
By \eqref{sum-gen} and \eqref{polar-repr} we get
\begin{equation} \label{r0-kernel}
R_0(\alpha, \lambda, x,y) = \frac{i}{4 } \sum_{m\in\Z} \ \, J_{|\alpha+m|}(r\sqrt{\lambda}\, ) \big(
J_{|\alpha+m|}(r'\sqrt{\lambda}\, ) + i\, Y_{|\alpha+m|}(r'\sqrt{\lambda}\, ) \big) \,
e^{im(\theta-\theta')}\,
\end{equation}
Using the identities
$$
Y_\nu(z) = \frac{J_\nu(z)\, \cos(\nu\pi) - J_{-\nu}(z)}{\sin(\nu \pi)}, \qquad
J_\nu(z) = (z/2)^\nu\, \sum_{k=0}^\infty \, \frac{(-1)^k\, z^{2k}}{4^k\, k!\ \Gamma(\nu+k+1)},
$$
where $\nu\not\in\Z$, see \cite[Eqs. 9.1.2, 9.1.10]{as}, and the well-known properties of the Gamma function, namely
\begin{equation} \label{gamma-1}
\Gamma(z+1)=z\, \Gamma (z), \qquad \Gamma(z)\, \Gamma(1-z) =  \frac{\pi}{\sin(\pi z)},
\end{equation}
we can further rewrite \eqref{r0-kernel} in the power-series in $\lambda$ as follows:
\begin{align} \label{serie}
& R_0(\alpha, \lambda, x,y)  = G_0(x,y)  +  G_1(x,y)\, \lambda^{|\alpha|}\, + G_2(x,y) \, \lambda^{|\alpha-\sgn\alpha|}\\
& \quad +  \frac {1}{4\pi}\, \sum_{m\in \Z} \, e^{im(\theta-\theta')}\,  \Big(\frac{r}{r'}\Big)^{|\alpha+m|} \sum_{k,n\geq 0, k+n>0}\, \frac{(\frac 14\, r'^2 )^k\, (-\frac 14\, r^2)^n\, \lambda^{k+n}}{n!\, k!\, (n+|\alpha+m|)\cdots |\alpha+m|\cdots (|\alpha+m|-k)} \nonumber \\
& \quad +  \frac {i-\cot(\pi |\alpha|)}{4^{1+|\alpha|}}\,  \sum_{k,n\geq 0, k+n>0}\, \frac{(-\frac 14\, r'^2 )^k\, (-\frac 14\, r^2)^n\, (r r')^{|\alpha|} \, \lambda^{k+n+|\alpha|}}{\Gamma( n+|\alpha|+1)\, \Gamma(k+|\alpha|+1)} \nonumber \\
& \quad +  \sum_{m\neq 0} \, e^{im(\theta-\theta')}\, \frac{i-\cot(\pi |\alpha+m|)}{4^{1+|\alpha+m|}}\, \sum_{k,n\geq 0, k+n >0}\, \frac{(-\frac 14\, r'^2 )^k\, (-\frac 14\, r^2)^n\, (r r')^{|\alpha+m|}\, \lambda^{k+n+|\alpha+m|}}{\Gamma( n+|\alpha+m|+1)\, \Gamma(k+|\alpha+m|+1)} , \nonumber
\end{align}
where $G_0(x,y)$ is given by \eqref{g0} and
\begin{align*}
G_1(x,y) & = \frac{i-\cot(\pi |\alpha|)}{4\, \Gamma^2(1+|\alpha|)}\,  \Big(\frac {r r'}{4}\Big) ^{|\alpha|}, \\
G_2(x,y)  &=  e^{i (\theta'-\theta) \sgn\alpha}\, \, \frac{ i-\cot(\pi |\alpha-\sgn\alpha|)}{4 \, \Gamma^2(1+|\alpha-\sgn\alpha|)}\,  \Big(\frac {r r'}{4}\Big)^{|\alpha-\sgn\alpha|}  .
\end{align*}
Denote by $G_j, \, j=1,2$  the respective integral operators generated by the kernels $G_j(x,y)$. The following technical Lemma will be needed later for the asymptotic expansion of $R_0(\alpha, \lambda)$.

\begin{lemma} \label{lem-aux}
Let $a>1$. Then
$$
\int_0^\infty\, e^{-s^4}\, s^{4a}\, s\, ds \, \leq\,  \frac{a}{4}\ \Gamma(a).
$$
\end{lemma}

\begin{proof}
With the substitution $s^4=t$ we obtain
\begin{align*}
\int_0^\infty\, e^{-s^4}\, s^{4a}\, s\, ds & = \frac 14\, \int_0^\infty\, e^{-t}\,  t^{a-\frac 12}\, dt = \frac 14\, \Gamma\left (a+\frac 12 \right) \leq \frac 14\, \Gamma(a+1) = \frac a4\, \Gamma(a),
\end{align*}
where we have used the first part of \eqref{gamma-1}.
\end{proof}

\begin{lemma} \label{lem-1}
The series of integral operators \eqref{serie} converges absolutely in $\B(\rho, \rho^{-1})$ and uniformly in $\lambda$ on compacts of $[0,\infty)$. In particular, we have
\begin{equation} \label{R0-taylor1}
R_0(\alpha, \lambda)= G_0 + G_1\, \lambda^{|\alpha|} + G_2\, \lambda^{|\alpha-\sgn\alpha|} + o(\lambda^{|\alpha-\sgn\alpha|}) \qquad \text{in} \  \ \B(\rho, \rho^{-1}).
\end{equation}
as $\lambda \to 0$.
Moreover, the operator $R_0(\alpha, \lambda)$ generated by the integral kernel \eqref{serie} can be differentiated in $\lambda$ on $(0,1)$ any number of times in the norm of  $\B(\rho, \rho^{-1})$ and it holds
\begin{equation} \label{R0-der}
R_0'(\alpha,\lambda) =  |\alpha|\,  G_1\,  \lambda^{|\alpha|-1} + o(\lambda^{|\alpha|-1}), \quad  R_0''(\alpha,\lambda) =  |\alpha| (|\alpha|-1) G_1\,  \lambda^{|\alpha|-2} + o(\lambda^{|\alpha|-2}),
\end{equation}
where $R_0'(\alpha,\lambda)$ and $R_0''(\alpha,\lambda)$ denote first and second derivative of $R_0(\alpha,\lambda)$ with respect to $\lambda$ in the norm of $\B(\rho, \rho^{-1})$.
\end{lemma}

\begin{proof}
From \eqref{r0-kernel} we see that $R_0(\alpha, \lambda, x,y)$ can be written in the form
$$
R_0(\alpha, \lambda, x,y) = \sum_{m\in\Z} R^m_0(\alpha, \lambda, r, r')\  e^{im(\theta-\theta')},
$$
where we can identify the $R^m_0(\alpha, \lambda, r, r')$ with the terms on the right hand side of \eqref{serie}. The above formula implies that
\begin{align}
\|R_0(\alpha, \lambda)\|^2_{\B(\rho, \rho^{-1})}  & = \sup_{m\in\Z}\, \|\rho^{-1} \, R^m_0(\alpha, \lambda)\, \rho^{-1} \|^2_{L^2(\R_+, rdr)\to L^2(\R_+, r dr)} \nonumber \\
& \leq \int_0^\infty \int_0^\infty |R^m_0(\alpha, \lambda, r, r')|^2 \, \rho(r)^{-2}\, \rho(r')^{-2}\, r\, r'\, dr dr',
\label{hs-norm}
\end{align}
where we have used the Hilbert-Schmidt norm on $L^2(\R_+, r dr)$.
It is now easily verified that the series of operators on the right hand side of \eqref{serie} converges absolutely in $\B(\rho, \rho^{-1})$. Indeed,
the square of the first double-series on the right hand side of \eqref{serie} can be estimates by Cauchy-Schwarz inequality as follows
\begin{align}
&\left( \sum_{k,n\geq 0, k+n>0}\, \frac{(\frac 14\, r'^2 )^k\, (-\frac 14\, r^2)^n\, \lambda^{k+n}}{n!\, k!\, (n+|\alpha+m|)\cdots |\alpha+m|\cdots (|\alpha+m|-k)}\right)^2 \leq  \label{hs-aux} \\
& \qquad\qquad \qquad\qquad \qquad\qquad \qquad\qquad \qquad \qquad \qquad
\leq    \sum_{k > 0} \frac{  (\sqrt{\lambda}\ r')^{4k} }{(k!)^2} \  \sum_{n\geq 0} \frac{  (\sqrt{\lambda}\ r)^{4n} }{(n!)^2} .\nonumber
\end{align}
Hence using Lemma \ref{lem-aux} we find that the Hilbert-Schmidt norm of the kernel on the left hand side of \eqref{hs-aux} is bounded from above by a constant times $\lambda^2\, e^{\lambda^2}$. The remaining terms in \eqref{serie} are treated in the same way. In view of \eqref{hs-norm} we thus conclude that \eqref{serie} converges absolutely in $\B(\rho, \rho^{-1})$ for any $\lambda \geq 0$ and that the convergence is uniform in $\lambda$ on any compact interval of $[0,\infty)$. The same argument applies to $\frac{ d^k}{d \lambda^k} R_0(\alpha, \lambda)$ for any $k\in\N$. This implies \eqref{R0-der}.
\end{proof}

\subsubsection*{\bf The perturbed resolvent}

Next we consider the full resolvent $R(\alpha, z) = (H-z)^{-1}$ of the operator $H=H_\alpha +V$ in $L^2(\R^2)$. Since the contribution to the second term of the expansion for $R_0(\alpha,\lambda)$ is different when $|\alpha| < 1/2$ and when $|\alpha| =1/2$, in order to simplify the notation in the following Lemmas we introduce the symbol
\begin{equation} \label{s-alpha}
\sigma_\alpha = 0 \quad \text{if} \quad |\alpha| < 1/2, \qquad \sigma_\alpha = 1  \quad \text{if} \quad |\alpha| =1/2.
\end{equation}

\begin{lemma} \label{lem-exp1}
Under Assumption \ref{ass-V} we have
\begin{equation} \label{exp-1}
(1+R_0(\alpha,\lambda)\, V)^{-1} = (1+G_0 V)^{-1}-(1+G_0 V)^{-1}\, (G_1+\sigma_\alpha\, G_2)\, V\, (1+G_0 V)^{-1}\, \lambda^{|\alpha|} + o(\lambda^{|\alpha|})
\end{equation}
as  $\lambda \to 0$ in $\B(\rho^{-1}, \rho^{-1})$.
\end{lemma}

\begin{proof}
The operators $G_0, G_1 V$ and $G_2 V$ are Hilbert-Schmidt, and therefore compact, from $L^2(\R^2,\rho^{-1})$ to $L^2(\R^2,\rho^{-1})$. Hence by \eqref{V-as2} the operator $1+ G_0 V$ is invertible in $\B(\rho^{-1}, \rho^{-1})$. On the other hand, in view of \eqref{R0-taylor1} we have
\begin{equation}
1+R_0(\alpha,\lambda)\, V = 1+G_0 V + (G_1+\sigma_\alpha\, G_2) V\, \lambda^{|\alpha|} + o(\lambda^{|\alpha|})
\end{equation}
in $\B(\rho^{-1}, \rho^{-1})$. Here we have used the fact that $V\in\B(\rho^{-1}, \rho)$. It follows that for $\lambda$ small enough the operator $1+R_0(\alpha,\lambda)\, V$ can be inverted and with the help of the
Neumann series we arrive at \eqref{exp-1}.
\end{proof}

\begin{remark} \label{dual}
The operator $V G_0$ is compact from $L^2(\R^2,\rho)$ to $L^2(\R^2,\rho)$. Hence
by equation \eqref{V-as2} and duality $1+V G_0$ is invertible in $\B(\rho, \rho)$.
\end{remark}

\begin{lemma} \label{lem-exp2}
Under Assumption \ref{ass-V} we have
\begin{equation} \label{exp-2}
R(\alpha,\lambda) = T_0 + T_1\, \lambda^{|\alpha|}+ o(\lambda^{|\alpha|}) \
\end{equation}
as $\lambda \to 0$ in $\B(\rho, \rho^{-1})$, where
\begin{equation}
T_0= (1+G_0 V)^{-1}\, G_0, \qquad T_1 = (1+G_0 V)^{-1}\, (G_1+\sigma_\alpha\, G_2)\, (1+V G_0 )^{-1} .
\end{equation}	
\end{lemma}

\begin{proof}
Since $1+R_0(\alpha,\lambda)\, V$ is invertible in $\B(\rho^{-1}, \rho^{-1})$ for $\lambda$ small enough, the
resolvent equation yields
\begin{equation} \label{2-res}
R(\alpha,\lambda)=(1+R_0(\alpha,\lambda)\, V)^{-1}\, R_0(\alpha,\lambda),
\end{equation}
which in combination with \eqref{R0-taylor1} and \eqref{exp-1} gives equation \eqref{exp-2} with
$$
T_1 = (1+G_0 V)^{-1}\, (G_1+\sigma_\alpha\, G_2) -(1+ G_0 V )^{-1}\, (G_1+\sigma_\alpha\, G_2) \, V\, (1+G_0 V)^{-1} \, G_0
$$
in $\B(\rho, \rho^{-1})$. To simplify the above expression for $T_1$ let us consider $u\in L^2(\R^2, \rho)$ and denote $f=G_0\, u \in L^2(\R^2, \rho^{-1})$. Since $(1+G_0 V)^{-1}\, f =g$ if and only if $f= g + G_0 V  g$, and $(1+V G_0 )^{-1} \, V f = V g$ if an only if
$ V f = Vg + V G_0 V g$, we find out that $V\, (1+G_0 V)^{-1} \, G_0 u = (1+V G_0 )^{-1} \, V G_0 u$. Hence the identity
\begin{equation}
V\, (1+G_0 V)^{-1} \, G_0 = (1+V G_0 )^{-1} \, V G_0
\end{equation}
holds on $L^2(\R^2, \rho)$, which implies that
$$
T_1 =  (1+G_0 V)^{-1}\, (G_1+\sigma_\alpha\, G_2)\, (1+V G_0 )^{-1}.
$$
\end{proof}

\begin{corollary} \label{cor-der-0}
As $\lambda\to 0$,
\begin{equation} \label{R-der}
R''(\alpha,\lambda) =  |\alpha| (|\alpha|-1) (1-T_0 V)\, G_1\, (1-T_0 V)\, \lambda^{|\alpha|-2} +o(\lambda^{|\alpha|-2})  \quad \text{in} \  \ \B(\rho, \rho^{-1}).
\end{equation}
\end{corollary}

\begin{proof}
We use the identities
\begin{align}
R'(\alpha,\lambda) & = (1- R(\alpha,\lambda)V) \, R_0'(\alpha,\lambda)\, (1-V R(\alpha,\lambda) ) \label{R'} \\
R''(\alpha,\lambda) & =  (1-R(\alpha,\lambda) V) \, R_0''(\alpha,\lambda)\, (1-V R(\alpha,\lambda) ) -2 R'(\alpha,\lambda) \, V \,  R_0'(\alpha,\lambda) (1-V R(\alpha,\lambda)), \label{R''}
\end{align}
which hold in $B(\rho, \rho^{-1})$ in view of the resolvent equation $R=R_0- R_0\, V R$. The claim now follows from Lemma \ref{lem-1} and Lemma \ref{lem-exp2}.
\end{proof}


\subsection{ High energy behavior}
Here we will use the operator norms in the weighted spaces $L^2(\R^2, w^s)$.
When $s=0$ we write $\|u\|_0$ instead of $ \|  u\|_{0,0}$.
We denote by $\B(s, s')$ the set of bounded operators from $L^2(\R^2, w^s)$ to $L^2(\R^2, w^{s'})$.  Obviously, for any  $K\in \B(s,s')$ we have
\begin{equation} \label{imbed}
 p' \leq s' \leq s \leq p \quad \Rightarrow \quad \| K \|_{\B(p,p')} \,  \leq \, \| K \|_{\B(s,s')}.
\end{equation}

\begin{lemma} \label{lem-cont}
If $s>1/2$, then $R_0(\alpha, \lambda)$ is continuous in $\lambda$ on $(0,\infty)$ in the norm of $\B(s,-s)$.
\end{lemma}

\begin{proof}
First we show that for any $s$ with $s \leq 1$ and any $z\in\C$ with Im$\, z\neq 0$ we have $R_0(\alpha, z) \in \B(s,s)$. Given $u\in L^2(\R^2, w^s)$, we have
\begin{align} \label{comm}
\| R_0(\alpha,z) \, u\|_{0,s} &\ \leq \| \   R_0(\alpha,z) \, w^s\, u\|_0 + \| [ R_0(\alpha,z), w^s]  \, u\|_0  \nonumber \\
&\  = \  \|  R_0(\alpha,z) \, w^s\, u\|_0 + \|  R_0(\alpha,z)\, [H_\alpha, \, w^s]\, R_0(\alpha,z)  \, u\|_0.
\end{align}
In the polar coordinates $H_\alpha$ acts as
$$
H_\alpha  = -\partial_r^2 -\frac 1r\partial_r +\frac{1}{r^2} (i\partial_\theta +\alpha)^2 .
$$
Hence when calculated on functions from $C_0^\infty(\R^2\setminus\{0\})$ the commutator $[H_\alpha, w^s]$ reads as
$$
[ H_\alpha, w^s]  = -s\, w(r)^{ s -2}\, (s r^2+2) - 2s\, r\, w(r)^{ s -2}\, \partial_r,
$$
where $r=|x|$. Since $s\leq 1$, the first term on the right hand side is bounded. Moreover, from \eqref{q-form} it follows that for for every $f \in L^2(\R^2)$
$$
\| \partial_r \, R_0(\alpha, z)\, f\|^2_{L^2(\R^2)} \, \leq\, |\, \lef f, \, R_0(\alpha,z)\, f \rig _{L^2(\R^2)} | + |z|\, \|R_0(\alpha, z)\, f\|^2_{L^2(\R^2)}.
$$
This in combination with \eqref{comm} and the fact that
\begin{equation}
\| R_0(\alpha, z)\|_{\B(0,0)}\, \leq\, \frac{1}{|{\rm Im}\, z|}
\end{equation}
gives
\begin{equation} \label{ss}
\| R_0(\alpha,z) \, u\|_{0,s} \leq \, C\, \Big( \frac{1}{|{\rm Im}\, z|}+  \frac{|z|}{|{\rm Im}\, z|^2}\Big)\, \|u\|_{0,s}.
\end{equation}
Now, let $\lambda, \lambda' \in \R$ and let $\eps>0$. Then by the resolvent equation
\begin{equation} \label{resolv}
\| (R_0(\alpha, \lambda+i\eps) - R_0(\alpha, \lambda'+i\eps ))\, u\|_{0,s} = |\lambda-\lambda'|  \, \| R_0(\alpha, \lambda+i\eps)\, R_0(\alpha, \lambda'+i\eps )\, u\|_{0,s}\, .
\end{equation}
Hence in view of \eqref{ss}, for $\eps$ small enough
\begin{align*}
& \| (R_0(\alpha, \lambda) - R_0(\alpha, \lambda')) u\|_{0,-s}  \leq  \| R_0(\alpha, \lambda) - R_0(\alpha, \lambda+i\eps )\|_{\B(s, -s)}  \|u\|_{0,s}  \\
& \quad +\| R_0(\alpha, \lambda') - R_0(\alpha, \lambda'+i\eps )\|_{\B(s, -s)}  \|u\|_{0,s} + \mathcal{O}(\eps^{-4}) \, |\lambda-\lambda' | (1+|\lambda|)(1+|\lambda'|)\, \|u\|_{0,s}
\end{align*}
Since the first two terms on the right hand side converge to zero as $\eps \to 0$ locally uniformly in $\lambda$ respectively $\lambda'$, see \eqref{it-limit}, this proves the continuity of $R_0(\alpha, \lambda)$ in $\lambda\in (0,\infty)$ for $s \in(1/2,\, 1]$.  By \eqref{imbed} the claim holds for all $s>1/2$.
\end{proof}

\begin{lemma} \label{lem-scale}
For any $\eps>0$ there exists $C_\eps$ such that for all $s , s' \geq 1/2 +\eps$ it holds
\begin{equation} \label{eq-scale}
\|R_0(\alpha,\lambda)\|_{\B(s,-s')}  \, \leq \,  C_\eps\, \lambda^{-\frac 12 +\eps} \qquad \forall\ \lambda\geq 1.
\end{equation}
\end{lemma}

\begin{proof}
Let $s = 1/2 +\eps$ and let $u\in L^2_s(\R^2)$. We define $u_\lambda(x) = u(x/\sqrt{\lambda})$. Since $R_0(\alpha,\lambda,x,y) = R_0(\alpha, 1, \sqrt{\lambda}\, x, \sqrt{\lambda}\, y)$, see equation \eqref{r0-kernel}, a simple change of variables gives
\begin{align*}
\| R_0(\alpha, \lambda)\, u\|^2_{0,-s} & \leq \lambda^{-3+s}\, \|R_0(\alpha, 1)\|^2_{\B(s, -s)}\, \| u_\lambda\|^2_{0,s}.
\end{align*}
On the other hand,
$\| u_\lambda\|^2_{0,s} \, \leq \, \lambda^{1+s}\, \| u\|^2_{0,s}$.
Hence
\begin{align*}
\| R_0(\alpha, \lambda)\, u\|^2_{0,-s} & \leq \lambda^{-1+2\eps}\, \|R_0(\alpha, 1)\|^2_{\B(s, -s)}\, \| u\|^2_{0,s}.
\end{align*}
In view of \eqref{imbed} this completes the proof.
\end{proof}

\noindent In order to obtain suitable estimates on the derivatives of $R_0(\alpha, \lambda)$, we need the following technical result.

\begin{lemma} \label{lem-comm}
Let $T = (i\nabla +A_{ab})\cdot x 
$. Assume that $s, s' > 5/2$. Then for any $\eps>0$ we have
\begin{equation} \label{l-infty}
[\, i T, \, R_0(\alpha, \lambda) ]  = \mathcal{O}(\lambda^{\eps}) \quad \text{in} \ \ \B(s,-s') \quad \text{as} \quad \lambda\to\infty.
\end{equation}
\end{lemma}

\begin{proof}
Without loss of generality we may assume that $0< \eps < 1/2$.
We introduce the notation $D_j = -\partial_j + i A_j, \, j=1,2$ so that $i\, T= D_1 x_1+ D_2 x_2$. Moreover, let
$$
X_j(\lambda) = w^{1-\frac{s'}{2}}\, D_j \, R_0(\alpha, \lambda), \qquad j=1,2.
$$
We observe that  $[ D_j, x_j]= -1$. Thus by Lemma \ref{lem-scale} in order to prove \eqref{l-infty} it suffices to show that there exists a constant $C$ such that for all $\lambda$ large enough and any $p> 1/2$, and $u\in L^2_p(\R^2)$ we have
\begin{align} \label{eq-1}
\| w^{-\frac{s'}{2}} (x_1\, D_1 +x_2\, D_2) \, R_0(\alpha, \lambda) u\|_0^2 & \, \leq 2\, \| X_1(\lambda)\, u\|_0^2  + 2\, \|  X_2(\lambda)\, u\|_0^2  \leq C\,  \lambda^{2\eps}\, \|u\|_{0,p}^2.
\end{align}
Note that $w^{\frac{s'}{2} -1}\, [ D_j, w^{1-\frac{s'}{2}}] = - w^{\frac{s'}{2} -1} \, \partial_j w^{1-\frac{s'}{2}}$ is a bounded function for $j=1,2$. Hence by the Cauchy-Schwarz inequality and \eqref{eq-scale} we get
\begin{align}
  \| X_j(\lambda)\, u\|_0^2  & = \lef w^{1-\frac{s'}{2}}\, R_0(\alpha, \lambda) u, \, [ D_j^*, w^{1-\frac{s'}{2}}] D_j R_0(\alpha, \lambda) \, u\rig_{L^2(\R^2)}  \label{x^2}\\
&  \quad + \lef  [w^{1-\frac{s'}{2}},  D_j]  \, R_0(\alpha, \lambda) u, w^{1-\frac{s'}{2}}\, D_j R_0(\alpha, \lambda) u\rig_{L^2(\R^2)} \nonumber \\
& \quad + \lef w^{1-\frac{s'}{2}}\, R_0(\alpha, \lambda)\, u, \,  w^{1-\frac{s'}{2}}\, D_j^* D_j R_0(\alpha, \lambda) u\rig_{L^2(\R^2)} \nonumber\\
&  \leq 2\, C_\eps \, \lambda^{-\frac 12+\eps}\,  \| X_j(\lambda)\, u\|_0 \, \|u\|_{0,p} + \lef w^{1-\frac{s'}{2}}\, R_0(\alpha, \lambda)\, u, \,  w^{1-\frac{s'}{2}}\, D_j^* D_j R_0(\alpha, \lambda) \, u\rig_{L^2(\R^2)}.\nonumber
\end{align}
In the application of \eqref{eq-scale} we used the assumption $s'-2 > 1/2$.  Now from the identity
$$
\sum_{j=1}^2\, D_j^* D_j\, R_0(\alpha, \lambda) = H_\alpha\, R_0(\alpha, \lambda) = 1 +\lambda\, R_0(\alpha, \lambda)
$$
and equations \eqref{eq-scale} and \eqref{x^2} we obtain
\begin{align*}
\sum_{j=1}^2\,  \| X_j(\lambda)\, u\|_0^2  &\,  \leq \,  C_\eps \, \lambda^{-\frac 12+\eps}\,  \sum_{j=1}^2\,  \| X_j(\lambda)\, u\|_0^2  + 2\, C_\eps\,  \lambda^{-\frac 12+\eps}\,\|u\|^2_{0,p} + C_\eps^2\, \lambda^{2 \eps}\, \|u\|_{0,p}^2.
\end{align*}
Hence taking $\lambda$ large enough so that $C_\eps\, \lambda^{-\frac 12+\eps} \leq 1/2$ and at the same time $\lambda >1$, we deduce from the last equation that
$$
\sum_{j=1}^2\,  \| X_j(\lambda)\, u\|_0^2  \, \leq \, 2\, (1+ C^2_\eps)\, \|u\|_{0,p}^2.
$$
This implies \eqref{eq-1}. To prove that $\| R_0(\alpha, \lambda)\, T \|_{\B(s, -s')} = \mathcal{O}(\lambda^{\eps})$, let $u\in C_0^\infty(\R^2\setminus\{0\})$ and write $v_j = x_j u$. From \eqref{eq-1} applied to $v_j$ with $p=s-2$ we deduce that
\begin{align} \label{eq-2}
\| w^{-\frac{s'}{2}} ( D_1\, R_0(\alpha, \lambda)\, v_1 + D_2\, R_0(\alpha, \lambda)\, v_2) \|_0^2 & \, \leq 2\, \sum_{j=1}^2 \| w^{-\frac{s'}{2}}\, D_j R_0(\alpha, \lambda) v_j\|_0^2  =  \mathcal{O}(\lambda^{2\eps})\, \|u\|_{0,s}^2
\end{align}
as $\lambda\to \infty$.
Since the operators $D_j$ commute with $R_0(\alpha,\lambda)$ on $C_0^\infty(\R^2\setminus\{0\})$, inequality \eqref{eq-2} implies that
$$
\| R_0(\alpha, \lambda)\, T \, u \|_{-s'}^2 = \mathcal{O}(\lambda^{2\eps})\, \|u\|_{0,s}^2 \quad \text{as} \ \  \lambda\to \infty
$$
for all $u\in C_0^\infty(\R^2\setminus\{0\})$. By density we conclude that $\| R_0(\alpha, \lambda)\, T \|_{\B(s, -s')} = \mathcal{O}(\lambda^{\eps})$.
\end{proof}

\begin{lemma} \label{lem-high-der}
Let $s,s' > 5/2$.
Let $R_0^{(k)}(\alpha,\lambda)$ denote the $k-$th derivative of $R_0(\alpha,\lambda)$ with respect to $\lambda$ in $\B(s, -s')$.  Then  for any $\eps>0$ there exists $C_{\eps,k}$ such that
\begin{equation} \label{der-high}
\|R_0^{(k)}(\alpha,\lambda)\|_{\B(s, -s')}  \, \leq \,  C_{\eps,k}\ \lambda^{-\frac {k+1}{2} +\eps} \qquad \forall\ \lambda\geq 1, \quad k=1,2.
\end{equation}
\end{lemma}

\begin{proof}
Let $T$ be the operator introduced in Lemma \ref{lem-comm}. Then the commutator $[\, i T, H_\alpha]$ calculated on $C_0^\infty(\R^2\setminus\{0\})$ gives
\begin{equation*}
[ \, i T, \, H_\alpha] = 2\, H_\alpha.
\end{equation*}
This together with the first resolvent equation $R_0(\alpha, z)- R_0(\alpha, \zeta) = (z-\zeta)R_0(\alpha, z)R_0(\alpha, \zeta),\, z, \zeta \in\C_+$, implies the identity
\begin{equation} \label{comm-1}
[\, i T, \, R_0(\alpha, z) ] =   -R_0(\alpha, z) \, [ \, i T, \, H_\alpha]\, R_0(\alpha, z) = -2 R_0(\alpha, z) - 2 z\, R_0'(\alpha, z).
\end{equation}
On the other hand, from Lemma \ref{lem-comm} it follows that $[\, i T, \, R_0(\alpha, \lambda)] \in \B(s, -s')$ for $\lambda \geq 1$.
Equation \eqref{comm-1}, considered in $\B(s, -s')$, thus can be thus extended to $z= \lambda \in \R, \, \lambda\geq 1$. This shows that $R_0'(\alpha, \lambda)$ is continuous in $\lambda$ on $(1,\infty)$ with respect to the norm of $\B(s, -s')$ and in view of Lemma \ref{lem-comm}
\begin{equation}  \label{r0'}
R_0'(\alpha,\lambda) \, = \mathcal{O}(|\lambda|^{-1+\eps}) \quad \text{in} \ \ \B(s,-s'), \quad \lambda\to\infty.
\end{equation}
Hence equation \eqref{der-high} is proven for $k=1$. To prove it for $k=2$ we consider the double commutator
$$
\big[ x, [x, R_0(\alpha,z)] \big] = |x|^2\, R_0(\alpha,z) -2 \, \sum_{j=1}^2\, x_j R_0(\alpha,z)\, x_j + R_0(\alpha,z)\, |x|^2
$$
and use the identity
\begin{equation} \label{comm-2}
z\, R''_0(\alpha, z) = -3 R_0'(\alpha,z) -\frac 14\, \big[ x, [x, R_0(\alpha,z)] \big],
\end{equation}
which follows by a direct calculation in the same way as equation \eqref{comm-1}. By using the fact that
\begin{align}
x  & \in \B(s, s-1) \qquad \forall \  s \in\R,  \label{x}
\end{align}
we extend equation \eqref{comm-2}, considered in $\B(s, -s')$, to $z = \lambda \in \R, \, \lambda \geq 1$. From \eqref{der-high} for $k=1$, \eqref{r0'}, and \eqref{eq-scale} we thus arrive at
 $$
R_0''(\alpha,\lambda) \, = \mathcal{O}(|\lambda|^{-\frac 32+\eps}) \quad \text{in} \ \ \B(s, -s'), \quad \lambda\to\infty.
$$
The claim now follows.
\end{proof}

\begin{corollary}  \label{cor-high-0}
For any $\eps>0$ and any $k=0,1,2$ there exist constants $C_{\eps,k}$ such that
\begin{equation}
\|R_0^{(k)}(\alpha,\lambda)\|_{\B(\rho, \rho^{-1})}  \, \leq \,  C_{\eps,k}\  \lambda^{-\frac {k+1}{2} +\eps} \qquad \forall\ \lambda\geq 1.
\end{equation}
\end{corollary}

\begin{proof}
Since for any $s>0$ there exists $C_s$ such that
$$
\|R_0(\alpha,\lambda)\|_{\B(\rho, \rho^{-1})}  \, \leq \,  C_s \, \|R_0(\alpha, \lambda)\|_{\B(s, -s)},
$$
the claim follows from Lemmata \ref{lem-scale}, \ref{lem-high-der}.
\end{proof}

\begin{proposition}  \label{prop-high}
Suppose that Assumption \ref{ass-V} is satisfied. Then, for any $k=0,1,2$  we have $R^{(k)}(\alpha,\lambda)\in \B(\rho, \rho^{-1})$ and
\begin{equation} \label{res-infty}
\|R^{(k)}(\alpha,\lambda)\|_{\B(\rho, \rho^{-1})}  \, = \,   \mathcal{O}(\lambda^{-\frac {k+1}{2} +0}) \qquad  \lambda\to \infty.
\end{equation}
\end{proposition}

\begin{proof}
Let us first show that the operator $H=H_\alpha +V$ has no positive eigenvalues. To this end assume that there exists $\lambda >0$ such that $H\, u= \lambda\, u$. Since $V$ is compactly supported, there exists $R>0$ such that $V=0$ outside the ball $B(0,R)$. We can thus use decomposition \eqref{sum-gen} to obtain
\begin{equation}
h_m \, f_m (r) = \lambda\, f_m(r) \qquad \forall \ m\in\Z, \quad \forall\ r > R,
\end{equation}
where
$$
f_m(r) = \frac{1}{2\pi} \, \lef \, e^{im \theta},\, u(r,\theta)\rig_{L^2(0,2\pi)}
$$
and $h_m$ is the operator associated with quadratic form $Q_m$ given by \eqref{qm}.
A direct calculation now shows that $f_m$ is a linear combination of the Bessel functions
$J_{|m +\alpha|}(r\sqrt{\lambda})$ and  $Y_{|m+\alpha|}(r\sqrt{\lambda})$. Since any non-trivial linear combination
of these functions does not belong to $L^2((R,\infty), r dr)$, see \cite[Eqs.9.2.1-2]{as}, we conclude that $f_m=0$ on $(R,\infty)$ and hence $u=0$ outside $B(0,R)$. By the unique continuation principle we then have $u=0$.
Hence $\lambda$ is not an eigenvalue.

Next, since $R_0(\alpha, \lambda)$ is Hilbert-Schmidt, and therefore compact, from $L^2(\R^2, \rho)$ to $L^2(\R^2, \rho^{-1})$, and $V\in \B(\rho^{-1}, \rho)$, we observe that $V R_0(\alpha, \lambda)$ is compact from $L^2(\R^2, \rho)$ to $L^2(\R^2, \rho)$. Moreover the fact that $H$ has no positive eigenvalues implies that $-1$ is not an eigenvalue of $V R_0(\alpha, \lambda)$ in $\B(\rho, \rho)$ for any $\lambda \in (0,\infty)$. Hence the operator $1+ V R_0(\alpha, \lambda)$ is invertible in $\B(\rho, \rho)$ and its inverse is continuous with respect to $\lambda$ in the norm of $\B(\rho, \rho)$. The latter follows from the continuity of $VR_0(\alpha, \lambda)$ in $\B(\rho, \rho)$. Now equation \eqref{2-res} implies that $R(\alpha,\lambda) \in \B(\rho, \rho^{-1})$ and that $R(\alpha,\lambda)$ is continuous in $\lambda\in (0,\infty)$.

Finally, since $\|V\, R_0(\alpha, \lambda)\|_{\B(\rho, \rho)} \to 0$ as $\lambda \to \infty$, by Corollary \ref{cor-high-0}, we conclude that $(1+ V R_0(\alpha, \lambda))^{-1}$ is uniformly bounded for $\lambda\to\infty$ in the norm of $\B(\rho, \rho)$.
In view of \eqref{2-res} and Corollary \ref{cor-high-0} we thus get
$$
\|R(\alpha,\lambda)\|_{\B(\rho, \rho^{-1})}  \, = \,   \mathcal{O}(\lambda^{-\frac {1}{2} +0}) \qquad  \lambda\to \infty.
$$
This proves \eqref{res-infty} for $k=0$. The proof for $k=1$ and $k=2$ now follows from \eqref{res-infty} with $k=0$, identities \eqref{R'}, \eqref{R''} and Corollary \ref{cor-high-0}.
\end{proof}

\begin{remark} \label{lap} The absence of positive eigenvalues, as a consequence of the limiting absorption principle, was studied in great generality in the seminal work \cite{ag} for non-magnetic Schr\"odinger operators.
\end{remark}

\section{\bf Time decay}
\label{sec-proofs}

\subsection{Proof of Theorem \ref{thm-2}} For the sake of brevity we will prove the statements of Theorem \ref{thm-2} only for $|\alpha| < 1/2$.
The proof in the case $|\alpha| = 1/2$ is completely analogous.  By the spectral theorem  and the Stone formula we have
\begin{equation} \label{sp-thm}
e^{-i t H}\, P_c = \int_0^\infty\, e^{-i t \lambda} \, E(\alpha, \lambda)\, d\lambda,
\end{equation}
where
\begin{equation} \label{stone}
E(\alpha, \lambda) = \frac{1}{2\pi i} \lim_{\eps\to 0+} (R(\alpha, \lambda+i\eps) -R(\alpha, \lambda-i\eps))=   \frac{1}{\pi}\, {\rm Im}\, R(\alpha, \lambda).
\end{equation}
Here the imaginary part is meant in the sense of quadratic forms.
Let $\chi \in C^\infty(0,\infty)$ be such that $\chi(x)=0$ for all $x$ large enough and $\chi(x)=1$ in a neighborhood of $0$. Since
$E''(\alpha, \lambda) \in L^1((\delta,\infty); \B(\rho, \rho^{-1}))$ for any $\delta >0$ in view of Proposition \ref{prop-high}, we can apply \cite[Lemma 10.1]{JK}, see Lemma \ref{lem-jk1} in Appendix \ref{app}, to obtain
\begin{equation}
\int_0^\infty\, e^{-i t \lambda} \, (1-\chi(\lambda))\, E(\alpha, \lambda)\, d\lambda = o(t^{-2}) \quad \text{as} \quad t\to \infty
\end{equation}
in $\B(\rho, \rho^{-1})$.
On the other hand, for $\lambda\to 0$ we have by \eqref{stone} and Lemma \ref{lem-exp2}
$$
E(\alpha, \lambda) = E_1\, \lambda^{|\alpha|} + E_2(\alpha,\lambda), \qquad  E_2(\alpha,\lambda) = o(\lambda^{|\alpha|}),
$$
where
$$
E_1= \frac{1}{\Gamma(1+|\alpha|)}\,  (1+G_0 V)^{-1}\, \mathcal{G}_1 \,(1+ V G_0 )^{-1}.
$$
Moreover, by Corollary \ref{cor-der-0} we know that $E''_2(\alpha,\lambda) = o(\lambda^{|\alpha|-2})$ as $\lambda\to 0$. Hence from \cite[Lemma 10.2]{JK}, see Lemma \ref{lem-jk1} in Appendix \ref{app}, applied to $\chi(\lambda)\, E_2(\alpha, \lambda)$ it follows that
$$
\int_0^\infty\, e^{-i t \lambda} \, \chi(\lambda)\, E_2(\alpha, \lambda)\, d\lambda = o (t^{-1-|\alpha|})  \qquad t\to \infty.
$$
Finally, since $\chi(\lambda)=1$ in a vicinity of zero, the Erdelyi's lemma about asymptotic expansion of Fourier integrals, see \cite{erd2} or \cite[p. 639]{z}, gives
$$
\int_0^\infty \, e^{-i t \lambda} \, \lambda^{|\alpha|}\, \chi(\lambda)\, d\lambda = (it)^{-1-|\alpha|}\, \Gamma(1+|\alpha|) + o (t^{-1-|\alpha|}) \qquad t\to \infty.
$$
Summing up, we have
$$
e^{-i t H}\, P_c  = (it)^{-1-|\alpha|}\, \Gamma(1+|\alpha|)\, E_1 + o (t^{-1-|\alpha|})
$$
in $\B(\rho, \rho^{-1})$. This completes the proof of \eqref{limit-3}. Equation \eqref{limit-4} follows in the same way with $E_1$ replaced by
$$
 \frac{1}{\Gamma(1+|\alpha|)}\,  (1+G_0 V)^{-1}\, (\mathcal{G}_1 + \mathcal{G}_2) \,(1+ V G_0 )^{-1}.
$$
Estimate \eqref{weight-estim} is now immediate since $\| \, e^{-i\, t H}\, P_c\, u \|_{\rho^{-1} }\, \leq\,  \|u\|_{\rho}$ for all $t>0$.

\begin{remark}
In principle it would be possible to extend the above analysis and obtain higher order terms in the asymptotic expansion for $e^{-i t H}\, P_c$ as $t\to\infty$.
\end{remark}


\subsection{Proof of Theorem \ref{thm-1}}   We first establish an explicit formula for $e^{-i\, tH_\alpha}(x,y)$. This has been done already in \cite{fffp}. For the sake of completeness we provide an alternative derivation.
To start with we consider the one-dimensional operators $h_m$ in $L^2(\R_+, r dr)$ associated with the closure of the quadratic form $Q_m$ given by equation \eqref{qm}.

\begin{lemma} \label{bessel}
Let $f\in L^2(\R_+, r dr)$ be compactly supported. Then for all $t>0$ we have
\begin{equation} \label{bessel-eq}
\big( e^{-i t \, h_m} f\big) (r)  = \frac{1}{2 i t}\, \, \int_0^\infty
I_{|m+\alpha|}\left(\frac{r r'}{2 i t}\right)\,
e^{-\frac{r^2+r'^2}{4it}}\, f(r')\, r'\, dr'.
\end{equation}
where $I_{|m+\alpha|}$ is the modified Bessel function of the first kind.
\end{lemma}

\begin{proof} Recall the integral representation
\begin{equation} \label{int-repr}
I_\nu(z)= \frac{z^\nu}{2^\nu\, \Gamma(\nu+\frac 12)\Gamma(\frac 12)}\,
\int_{-1}^1\, (1-s^2)^{\nu-\frac 12}\, e^{zs}\, ds, \qquad z\in\C,
\end{equation}
see \cite[eq.9.6.18]{as}. Consider now the operator
\begin{equation} \label{lbeta}
L_m  = U\, h_m\, U^{-1} \qquad \text{in \, \, \, } L^2(\R_+,
dr),
\end{equation}
where $U: L^2(\R_+, r\, dr)\to L^2(\R_+, dr)$ is the unitary mapping
acting as $(U f)(r) = r^{1/2} f(r)$. Note that $L_m $ is subject
to Dirichlet boundary condition at zero and that it coincides with
the Friedrichs extension of the differential operator
$$
-\frac{d^2}{dr^2}\, +\,  \frac{\mu^2-\frac 14}{r^2}
$$
defined on $C_0^\infty(\R_+)$. From \cite[Sect. 5]{k11} we know that
\begin{equation} \label{diag}
W_m \, L_m \, W_m ^{-1}\, \varphi(p) = p\,
\varphi(p), \quad \varphi\in W_m (D(L_m )),
\end{equation}
where the mappings $W_m ,\, W_m ^{-1}: L^2(\R_+) \to L^2(\R_+)$ given
by
\begin{equation} \label{fb}
(W_m \, u)(p) = \int_0^\infty u(r)\sqrt{r}\,
J_{|m+\alpha|}(r\sqrt{p})\, dr, \quad (W_m ^{-1} \varphi)(r) =
\frac 12\, \int_0^\infty \varphi(p)\sqrt{r}\,
J_{|m+\alpha|}(r\sqrt{p})\, dp,
\end{equation}
and defined initially on $C_0^\infty(\R_+)$ extend to unitary operators
on $L^2(\R_+)$. By \cite[Thm.3.1]{T}
\begin{equation} \label{limit}
e^{-i t L_m }\, g  = \lim_{\eps\to 0+} e^{-(\eps+i t)\, L_m }\, g, \qquad g = U f.
\end{equation}
In view of \eqref{diag} we thus get
\begin{align}
 \lim_{\eps\to 0+} \big( e^{-(\eps+i t) L_m }\, g\big)(r) & =  \lim_{\eps\to 0+}  \big( W_m ^{-1}\, e^{-(\eps+i t)
p}\, W_m \, g\big)(r) \nonumber \\
& = \lim_{\eps\to 0+}  \frac 12\, \int_0^\infty \sqrt{r r'}
\int_0^\infty e^{-(\eps+i t) p}  J_{|m+\alpha|}(r\sqrt{p})
J_{|m+\alpha|}(r'\sqrt{p})\,  dp\,  g(r')
\, dr' \nonumber \\
& = \lim_{\eps\to 0+} \frac{1}{2 (\eps+i t) }\, \int_0^\infty \sqrt{r r'}\, \,
I_{|m+\alpha|}\Big(\frac{r r'}{2(\eps+i t)}\Big)\,
e^{-\frac{r^2+r'^2}{4(\eps+i t)}}\, g(r')\, dr'.  \label{epsilon}
\end{align}
where we have used \cite[Eq.4.14(39)]{erde} to evaluate the
$p-$integral. Moreover, from \eqref{int-repr} it follows that
\begin{equation} \label{I-estim}
\Big | I_{|m+\alpha|}\Big(\frac{r r'}{2(\eps+i t)}\Big ) \, e^{-\frac{r^2+r'^2}{4(\eps+i t)}} \Big | \leq \, C\ \Big | \frac{r r' }{t }\Big |^{|m+\alpha|}
\end{equation}
for some constant $C >0$. On the other hand, by assumptions on $f$ we have $ (r')^{\frac 12+|m+\alpha|} g(r') =  (r')^{1+|m+\alpha|} f(r') \in L^1(\R_+)$. Hence using the Lebesgue dominated theorem and \eqref{I-estim} we can interchange
the limit with the integral in \eqref{epsilon} and use \eqref{limit} to get
\begin{align*}
 \big( e^{-i t L_m }\, g\big) (r) & = \frac{1}{2 i t }\, \int_0^\infty \sqrt{r r'}\, \,
I_{|m+\alpha|}\left(\frac{r r'}{2 i t}\right)\,
e^{-\frac{r^2+r'^2}{4 i t}}\, g(r')\, dr'.
\end{align*}
Equation \eqref{bessel-eq} now follows in view of \eqref{lbeta} and the identity $g(r') = \sqrt{r'}\, f(r')$.
\end{proof}

\noindent Now let $u\in L^2(\R^2)$ have compact support.
Since the projector $\Pi_m$ commutes with $h_m\otimes \mbox{id}$, see section
\ref{ss:pwave}, from \eqref{sum-gen} and \eqref{bessel-eq} we obtain
\begin{equation} \label{hk-gen}
\big(e^{-i\, tH_\alpha}\, u\big) (r,\theta) =   \frac{1}{4 \pi i t}\, \sum_{m\in\Z}\,  \int_0^\infty\, \int_0^{2\pi}
I_{|m+\alpha|}\left(\frac{r r'}{2 i t}\right)\, e^{im(\theta-\theta')}\,
e^{-\frac{r^2+r'^2}{4it}}\, u(r',\theta')\, r'\, dr'\, d\theta'.
\end{equation}
Moreover, in view of the integral representation \eqref{int-repr} the series
$
\sum_{m\in\Z}\,
I_{|m+\alpha|}\big(\frac{r\, r'}{2 i t}\big)\, e^{im(\theta-\theta')}
$
converges, for a fixed $r$, uniformly in $r'$ on compacts of $\R_+$. We can thus interchange the sum and the integral in \eqref{hk-gen} to conclude that \eqref{L2-kernel} holds true with
\begin{equation} \label{ab-prop}
e^{-i\, t H_\alpha}(x,y)=  \frac{1}{4 \pi i\, t}\, e^{-\frac{r^2+r'^2}{4it}} \, \sum_{m\in\Z}\,
I_{|m+\alpha|}\left(\frac{r r'}{2 i t}\right)\, e^{im(\theta-\theta')}\, .
\end{equation}
Assume now that $|\alpha| < 1/2$. With the help of \eqref{int-repr} it is easily seen that
$$
\lim_{t\to\infty}\, t^{|\alpha|}\, e^{-\frac{r^2+r'^2}{4it}} \, \sum_{m\neq 0}\,
I_{|m+\alpha|} \left(\frac{r r'}{2 i t}\right)\, e^{im(\theta-\theta')} = 0.
$$
Equation \eqref{limit-1} now follows from the identity
$$
\lim_{z\to 0} z^{-\nu}\,  I_\nu(z) = \frac{1}{2^\nu\, \Gamma(1+\nu)},
$$
see \cite[eq. 9.6.10]{as}. The proof for $|\alpha| =1/2$ follows the same line. To prove \eqref{point-upperb} we recall that by \cite[Corollary 1.7]{fffp} there exists a constant $C_0$ such that
\begin{equation} \label{diamag}
\sup_{x,y\in\R^2} | e^{-i\, t H_\alpha}(x,y) |\, \leq \, \frac{C_0}{t} \qquad \forall\ t >0.
\end{equation}
Let us define
$$
\Omega_1 := \{(x,y)\in\R^4\, :\,  |x|\, |y| < t\} ,\qquad  \Omega_2 := \{ (x,y)\in\R^4\, :\,  |x|\, |y| \geq  t\}, \qquad z: = \frac{|x|\, |y|}{t}.
$$
Since $|\alpha| \leq 1/2$, from \eqref{diamag}, \eqref{ab-prop} and \eqref{int-repr} we deduce that
\begin{align*}
& \sup_{(x,y)\in\R^4} \Big( \frac{t}{|x|\, |y|} \Big)^{|\alpha|}\, |\, e^{-i\, t H_\alpha}(x,y) |  \\
& \qquad = \max\Big\{ \sup_{(x,y)\in\Omega_1} \Big( \frac{t}{|x|\, |y|} \Big)^{|\alpha|}\, |e^{-i\, t H_\alpha}(x,y) |\, ,\, \sup_{(x,y)\in\Omega_2} \Big( \frac{t}{|x|\, |y|} \Big)^{|\alpha|}\, |e^{-i\, t H_\alpha}(x,y) | \Big\} \\
& \qquad \leq\, t^{-1}\, \max\Big\{ C_1\, \sup_{0<z <1} \, \sum_{m\in\Z} \frac{z^{|m+\alpha|-{|\alpha|}}}{2^{|m+\alpha|}\, \Gamma(|m+\alpha|+1/2)}\ , \ C_0 \Big\}  = C_2\, t^{-1},
\end{align*}
where $C_1$ and $C_2$ are positive constants. This implies \eqref{point-upperb} and completes the proof of Theorem \ref{thm-1}.

\smallskip

\begin{remark}
By using the relation $I_\nu (z) = e^{-\frac 12\, i \pi z }\, J_\nu(i z)$,
see \cite[eq.9.6.3]{as}, it is easily seen that equation \eqref{ab-prop} agrees with the expression for the Aharonov-Bohm propagator found  in \cite[Thm. 1.3]{fffp}.
\end{remark}

\medskip


\appendix

\section{}
\label{app}

\noindent For reader's convenience we recall below the results obtained in \cite[Lemma 10.1]{JK}  and  \cite[Lemma 10.2]{JK} regarding a function $F:\R_+ \to \BB$, where $\BB$ is an arbitrary Banach space.

\begin{lemma}[Jensen-Kato] \label{lem-jk1}
Suppose that $F(\lambda)=0$ in a neighborhood of zero and that $F'' \in L^1(\R_+; \BB)$. Then
$$
 \int_0^\infty\, e^{-i t \lambda} \, F(\lambda)\, d\lambda = o(t^{-2}) \qquad \text{as} \quad t\to \infty \quad \text{in} \ \  \BB.
$$
\end{lemma}

\medskip

\begin{lemma}[Jensen-Kato] \label{lem-jk2}
Suppose that $F(0)=0$, $F(\lambda)=0$ for $\lambda$ large enough and that $F'' \in L^1((\delta,\infty); \BB)$ for any $\delta >0$. Assume moreover that  $F''(\lambda) = o(\lambda^{\beta-2})$ as $\lambda \to 0$ for some $\beta\in (0,1)$ . Then
$$
 \int_0^\infty\, e^{-i t \lambda} \, F(\lambda)\, d\lambda = o(t^{-1-\beta}) \qquad \text{as} \quad t\to \infty \quad \text{in} \ \  \BB.
$$
\end{lemma}

\medskip

\begin{remark}
\cite{JK} provides more general versions of the above results. Here we limit ourselves to particular situations which suit our purposes.
\end{remark}




\medskip

\bibliographystyle{amsalpha}

\end{document}